\documentclass[11pt]{article}

\usepackage[margin=1in]{geometry}
\usepackage{amsmath,amssymb,amsthm,mathtools}
\usepackage{mathrsfs}
\usepackage{hyperref}
\usepackage{xcolor}
\usepackage{enumitem}

\hypersetup{
	colorlinks=true,
	linkcolor=blue!60!black,
	citecolor=blue!60!black,
	urlcolor=blue!60!black
}

\newtheorem{theorem}{Theorem}[section]
\newtheorem{proposition}[theorem]{Proposition}
\newtheorem{lemma}[theorem]{Lemma}
\newtheorem{corollary}[theorem]{Corollary}
\theoremstyle{remark}
\newtheorem{remark}[theorem]{Remark}

\newcommand{\Mn}{\mathbb M_n}

\newcommand{\Pn}{\mathbb P_n}
\newcommand{\tr}{\operatorname{Tr}}
\newcommand{\diag}{\operatorname{diag}}
\newcommand{\rank}{\operatorname{rank}}

\newcommand{\one}{\mathbf 1}
\newcommand{\ui}[1]{\left\lVert #1 \right\rVert_{\mathrm{ui}}}

\newcommand{\pospart}[1]{\left(#1\right)_+}

\title{Heron-Wasserstein majorization inequalities for spectral and Kubo-Ando geometric means}

\author{
	Trung Dung Vuong\thanks{Department of Mathematics, High School for the Gifted, Vietnam National University Ho Chi Minh City, Ho Chi Minh City, Vietnam. Email: \href{mailto:vtdung@ptnk.edu.vn}{\texttt{vtdung@ptnk.edu.vn}}.}
	\and
		Anh Thi Nguyen \thanks{Faculty of Mathematics and Computer Science, University of Science, Vietnam National University Ho Chi Minh City, Ho Chi Minh City, Vietnam. Email: \href{mailto:vtdung@ptnk.edu.vn}{\texttt{nathi@hcmus.edu.vn}}.}
	\and
		Trung Hoa Dinh\thanks{Department of Mathematics and Statistics, Troy University, Troy, Alabama 36079, USA. Email: \href{mailto:thdinh@troy.edu}{\texttt{thdinh@troy.edu}}.}
}

\date{May 2026}

\begin{document}
\maketitle
 
\begin{abstract}
	We prove sharp Heron-type majorization inequalities for two quadratic matrix
	expressions associated with the spectral and Kubo-Ando geometric means.  For
	the spectral geometric mean cross term, we show that
	\[
	\lambda\bigl(a^2A+b^2B+c(A\natural B)\bigr)
	\prec_w
	\lambda\bigl(W_{a,b}(A,B)\bigr),
	\qquad 0\le c\le 2ab,
	\]
	where \(W_{a,b}(A,B)\) is the weighted Bures-Wasserstein expression.  The
	coefficient \(2ab\) is sharp, and at this endpoint the weak majorization
	becomes majorization.
	
	For the Kubo-Ando geometric mean, we prove the direct comparison
	\[
	\lambda\bigl(a^2A+b^2B+2ab(A\#B)\bigr)
	\prec_w
	\lambda\bigl(W_{a,b}(A,B)\bigr).
	\]
	This settles, in the two-variable setting, Bhatia's question of whether the
	Heron-type norm inequality of Bhatia-Lim-Yamazaki admits a
	weak-majorization refinement.  More precisely, we prove
	\[
	\lambda\bigl(a^2A+b^2B+2ab(A\#B)\bigr)
	\prec_w
	\lambda\bigl((aA^{1/2}+bB^{1/2})^2\bigr),
	\]
	and consequently obtain the corresponding inequality for all unitarily
	invariant norms.
\end{abstract}

 
 \section{Introduction}
 \label{sec:introduction}
 Let \(\Pn\) denote the cone of \(n\times n\) complex positive definite matrices.
 For \(A,B\in\Pn\) and \(0\le t\le1\), the weighted Kubo-Ando geometric mean is
 \[
 A\#_tB
 =
 A^{1/2}(A^{-1/2}BA^{-1/2})^tA^{1/2}.
 \]
 Another noncommutative extension of the scalar geometric mean is the spectral
 geometric mean.  In its weighted form, it is given by
 \[
 A\natural_t B
 =
 (A^{-1}\#B)^t A (A^{-1}\#B)^t.
 \]
 The midpoint cases are
 \[
 A\#B=A\#_{1/2}B,
 \qquad
 A\natural B=A\natural_{1/2}B
 =(A^{-1}\#B)^{1/2}A(A^{-1}\#B)^{1/2}.
 \]
 The Kubo-Ando mean is the metric geometric mean in the theory of operator means
 \cite{Bhatia2007,KuboAndo1980}; the spectral mean was introduced by Fiedler and
 Ptak \cite{FiedlerPtak1997} and has been developed further in connection with
 spectral and log-majorization relations \cite{GanLiuTam2021,LeeLim2007}.
 Related weighted spectral and \(F\)-type means, including the family
 \[
 F_t(A,B)
 =
 (A^{-1}\#_tB)^{1/2}A^{2-2t}(A^{-1}\#_tB)^{1/2},
 \]
 have also been studied recently
 \cite{DinhTamVuong2024,GanKimMer2026,JeongKimTam2025}.
 
 The present paper does not study the full weighted families.  Instead, it
 focuses on the midpoint cross terms \(A\#B\) and \(A\natural B\) inside
 quadratic Heron-type expressions.  The comparison object is the weighted
 Bures-Wasserstein expression
 \[
 W_{a,b}(A,B)
 =
 a^2A+b^2B+ab\bigl((AB)^{1/2}+(BA)^{1/2}\bigr),
 \qquad a,b\ge0.
 \]
 For \(a=1-t\) and \(b=t\), this is the two-variable Bures-Wasserstein geodesic
 point
 \[
 A\diamond_tB
 =
 (1-t)^2A+t^2B
 +t(1-t)\bigl((AB)^{1/2}+(BA)^{1/2}\bigr),
 \]
 which is related to the Bures-Wasserstein geometry of positive definite
 matrices \cite{BhatiaJainLimDistance2019,BhatiaJainLimIneq2019}.
 
The question studied here is therefore additive and Heron-type: how do the
quadratic expressions obtained by inserting \(A\natural B\) or \(A\#B\) as the
cross term compare with the corresponding Bures-Wasserstein expression?  The
inequalities below have sharp coefficients and are not formal consequences of
the known order, near-order, or log-majorization comparisons for Wasserstein and
spectral means \cite{DumitruFranco2024,GanHuang2024,GanKim2024}. 
 
This additive nature is essential: adding the common term \(a^2A+b^2B\) is not
compatible with multiplicative or order-type comparisons.  Thus the Heron-type
inequalities below require arguments different from the known comparisons
between Wasserstein and spectral means.

With this notation, our first main result is the sharp spectral-Heron comparison
\begin{equation}
	\label{eq:intro-spectral-main}
	\lambda\bigl(a^2A+b^2B+c(A\natural B)\bigr)
	\prec_w
	\lambda\bigl(W_{a,b}(A,B)\bigr),
	\qquad 0\le c\le 2ab.
\end{equation}
When \(c=2ab\), the weak majorization in
\eqref{eq:intro-spectral-main}  becomes majorization.  The proof reduces,
after a Riccati change of variables, to a Schur multiplier estimate by a
positive semidefinite matrix with diagonal bounded by one; at the sharp
coefficient \(c=2ab\), this multiplier is a correlation matrix of rank at most
three.

Our second main result treats the Kubo-Ando Heron expression
\[
H^{\#}_{a,b}(A,B)=a^2A+b^2B+2ab(A\#B).
\]
In the normalized case \(a+b=1\), this is the corresponding two-variable Heron
mean.  We prove the direct weak majorization
\begin{equation}
	\label{eq:intro-kubo-main}
	\lambda\bigl(H^{\#}_{a,b}(A,B)\bigr)
	\prec_w
	\lambda\bigl(W_{a,b}(A,B)\bigr).
\end{equation}
This comparison uses a nonlinear pinching principle: if \(C\in\Pn\) and
\(R,S\in\Pn\) commute with \(R+S=I\), then
\begin{equation}
	\label{eq:intro-nonlinear-pinching}
	\lambda\bigl(RCR+SCS+2(RCR\#SCS)\bigr)
	\prec_w
	\lambda(C).
\end{equation}
After proving both Heron-Wasserstein comparisons, we show that the two Heron
expressions are themselves incomparable in weak majorization.  Hence the
Kubo-Ando result is not a formal consequence of the spectral one.

The nonlinear pinching principle also yields a weak-majorization strengthening
of the two-variable Heron form of a question of Bhatia, Lim, and Yamazaki
\cite{BhatiaLimYamazaki2016}.  They considered norm comparisons between
Kubo-Ando power means and their non-Kubo-Ando extensions.  In the
two-variable Heron case, the question becomes whether
\[
\ui{A+B+2(A\#B)}
\le
\ui{A+B+A^{1/2}B^{1/2}+B^{1/2}A^{1/2}}
\]
holds for every unitarily invariant norm.  Bhatia, Lim, and Yamazaki verified
this inequality for the Schatten norms \(p=1,2,\infty\), and Dinh, Dumitru, and
Franco subsequently proved the Schatten \(p\)-norm version for every
\(1\le p\le\infty\) \cite{DinhDumitruFranco2017}.  Related Schatten-norm and
determinant inequalities for matrix Heron means were obtained in
\cite{Dinh2017Heron}, and complementary inequalities related to the same
question were studied in \cite{GhabriesAbbasMouradAssi2023}.  Here we prove the
stronger weak-majorization statement
\begin{equation}
	\label{eq:intro-bly-main}
	\lambda\bigl(a^2A+b^2B+2ab(A\#B)\bigr)
	\prec_w
	\lambda\bigl((aA^{1/2}+bB^{1/2})^2\bigr).
\end{equation}
By Ky Fan dominance, \eqref{eq:intro-bly-main} implies the two-variable Heron
inequality for all unitarily invariant norms.  Thus, in this two-variable Heron
case, it upgrades the known Schatten-norm result to a weak-majorization
inequality.   

The paper is organized as follows.  Section~\ref{sec:preliminaries} fixes
notation and records a Riccati parametrization used repeatedly.
Section~\ref{sec:spectral-schur} proves the Schur multiplier mechanism and the
sharp spectral-Heron-Wasserstein comparison.
Section~\ref{sec:nonlinear-pinching} proves the nonlinear pinching principle,
the direct Kubo-Ando Heron-Wasserstein comparison, and the incomparability of
the two Heron expressions.
Section~\ref{sec:bly} proves the weak-majorization form of the two-variable
Bhatia-Lim-Yamazaki Heron inequality. 
 \section{Preliminaries}
 \label{sec:preliminaries}
 
For a Hermitian matrix $T$, let
\[
\lambda(T)=\bigl(\lambda_1^\downarrow(T),\ldots,
\lambda_n^\downarrow(T)\bigr)
\]
be the eigenvalue vector arranged in decreasing order.  For a real vector \(x\),
let \(x^\downarrow\) denote its decreasing rearrangement.  For
\(x,y\in\mathbb R^n\), we write \(x\prec_w y\) if
\[
\sum_{j=1}^k x_j^\downarrow\le
\sum_{j=1}^k y_j^\downarrow,
\qquad 1\le k\le n,
\]
and \(x\prec y\) if, in addition, equality holds for \(k=n\).  For positive vectors \(x,y\in(0,\infty)^n\), we say that \(x\) is
log-majorized by \(y\), and write \(x\prec_{\log}y\), if
\[
\prod_{j=1}^k x_j^\downarrow\le
\prod_{j=1}^k y_j^\downarrow,
\qquad 1\le k<n,
\qquad
\prod_{j=1}^n x_j=\prod_{j=1}^n y_j.
\]
This is stronger than weak majorization: \(x\prec_{\log}y\) implies
\(x\prec_w y\).
 
We shall use the following standard Ky Fan variational formulas; see, for
example, \cite{Bhatia1997}.  If $Y$ is Hermitian, then
 \[
 \sum_{j=1}^k\lambda_j^\downarrow(Y)
 =
 \max_{\rank E=k}\tr(EY),
 \]
 where the maximum is over orthogonal projections of rank $k$.  For a Hermitian
 matrix $H$, we write
 \[
 H_+=\frac{|H|+H}{2}
 \]
 for its positive part.  If \(Y\ge0\), then the preceding Ky Fan sum also admits the elementary
 threshold representation
 \begin{equation}
 	\label{eq:kyfan-threshold-formula}
 	\sum_{j=1}^k\lambda_j^\downarrow(Y)
 	=
 	\min_{t\ge0}\left\{kt+\tr\pospart{Y-tI}\right\}.
 \end{equation}
 Indeed, if \(\mu_1\ge\cdots\ge\mu_n\ge0\) are the eigenvalues of \(Y\), then the
 right-hand side equals
 \[
 \min_{t\ge0}\left\{kt+\sum_{j=1}^n(\mu_j-t)_+\right\},
 \]
 and the minimum is attained for any
 \(t\in[\mu_{k+1},\mu_k]\), with the convention \(\mu_{n+1}=0\).
 The Ky Fan dominance principle says that weak majorization of singular-value
 vectors is equivalent to domination for all unitarily invariant norms; see, for
 example, \cite[Chapter~IV]{Bhatia1997}.  Since all matrices compared below are
 positive semidefinite, their singular values are their eigenvalues.

 When norm inequalities are stated, \(\ui{\cdot}\) denotes an arbitrary
 unitarily invariant norm on \(\Mn\).  Thus an inequality of the form
 \[
 \ui{X}\le \ui{Y}
 \]
 means that the inequality holds for every unitarily invariant norm.

 We shall use the standard comparison
 \begin{equation}
 \label{eq:kubo-spectral-logmaj}
 \lambda(P\#Q)\prec_{\log}\lambda(P\natural Q),
 \qquad P,Q\in\Pn,
 \end{equation}
 proved in \cite[Theorem~2.7]{GanLiuTam2021}.  In particular,
 \begin{equation}
 \label{eq:kubo-spectral-trace}
 \tr(P\#Q)\le \tr(P\natural Q),
 \qquad P,Q\in\Pn.
 \end{equation}
 
 For $A,B\in\Pn$, the product $AB$ has positive spectrum, and $(AB)^{1/2}$ denotes
 its principal square root.  Although $(AB)^{1/2}$ need not be Hermitian, the sum
 $(AB)^{1/2}+(BA)^{1/2}$ is Hermitian.  Indeed,
 $((AB)^{1/2})^*$ is a square root of $BA$ whose spectrum is contained in
 $(0,\infty)$; by uniqueness of the principal square root,
 $((AB)^{1/2})^*=(BA)^{1/2}$.
 
 For $0\le t\le1$, we write
 \[
 A\diamond_tB:=W_{1-t,t}(A,B).
 \]
 Thus $A\diamond_tB$ is the two-variable Bures-Wasserstein geodesic point with
 weights $1-t$ and $t$.
 
 \begin{lemma}
 \label{lem:riccati-parametrization}
 Let $A,B\in\Pn$ and put
 \[
 X=A^{-1}\#B.
 \]
 Then $X$ is the unique positive definite solution of the Riccati equation
 $XAX=B$.  Moreover,
 \[
 (AB)^{1/2}=AX,
 \qquad
 (BA)^{1/2}=XA,
 \]
 and, for every $a,b\ge0$,
 \begin{equation}
 	\label{eq:W-Riccati-form}
 	W_{a,b}(A,B)
 	:=a^2A+b^2B+ab\bigl((AB)^{1/2}+(BA)^{1/2}\bigr)
 	=(aI+bX)A(aI+bX).
 \end{equation}
 \end{lemma}
 
 \begin{proof}
 The Riccati identity $XAX=B$ is a standard property of $A^{-1}\#B$.  Then
 $(AX)^2=A(XAX)=AB$ and $(XA)^2=(XAX)A=BA$.  The matrix $AX$ is similar to the
 positive definite matrix $A^{1/2}XA^{1/2}$, and $XA$ is similar to the same
 matrix.  Hence both $AX$ and $XA$ have spectrum contained in $(0,\infty)$, so by
 uniqueness of the principal square root,
 \[
 (AB)^{1/2}=AX,
 \qquad
 (BA)^{1/2}=XA.
 \]
 Formula \eqref{eq:W-Riccati-form} follows immediately from $XAX=B$.
 \end{proof}
 
 \section{Spectral Heron means and Schur majorization}
 \label{sec:spectral-schur}
 
 For $a,b\ge0$ and $c\ge0$, define the coefficient-perturbed spectral Heron
 expression
 \[
 H^{\natural,c}_{a,b}(A,B)
 :=a^2A+b^2B+c(A\natural B).
 \]
 The sharp spectral Heron expression is
 \[
 H^\natural_{a,b}(A,B):=H^{\natural,2ab}_{a,b}(A,B)
 =a^2A+b^2B+2ab(A\natural B).
 \]
 The key tool for comparing $H^{\natural,c}_{a,b}$ with $W_{a,b}$ is a Schur
 multiplier contraction.
 
\begin{lemma}
	\label{lem:schur-subcorrelation}
	Let \(\Gamma=(\gamma_{ij})\in\Mn\) be positive semidefinite.
	\begin{enumerate}[label=\textup{(\roman*)}]
		\item If \(\gamma_{ii}=1\) for all \(i\), then for every Hermitian matrix
		\(M\),
		\[
		\lambda(\Gamma\circ M)\prec\lambda(M),
		\]
		where \(\circ\) denotes the Schur product.
		
		\item If \(0\le\gamma_{ii}\le1\) for all \(i\), then for every positive
		semidefinite matrix \(M\),
		\[
		\lambda(\Gamma\circ M)\prec_w\lambda(M).
		\]
	\end{enumerate}
\end{lemma}

\begin{proof}
	The correlation-matrix case in \textup{(i)} is the classical Schur-product
	majorization theorem of Bapat and Sunder \cite{BapatSunder1985}; we recall a
	short proof for completeness.  Let
	\[
	\Phi_\Gamma(M)=\Gamma\circ M.
	\]
	Since \(\Gamma\ge0\), the Schur map \(\Phi_\Gamma\) is completely positive.  Its
	Hilbert-Schmidt adjoint is \(\Phi_\Gamma^*=\Phi_{\overline{\Gamma}}\).  If
	\(\gamma_{ii}=1\) for all \(i\), then
	\[
	\Phi_\Gamma(I)=I,
	\qquad
	\Phi_\Gamma^*(I)=I,
	\]
	so \(\Phi_\Gamma\) is unital and trace preserving.  Let \(M\) be Hermitian and
	let \(E\) be an orthogonal projection of rank \(k\).  Put
	\[
	F=\Phi_\Gamma^*(E).
	\]
	Then
	\[
	0\le F\le I,
	\qquad
	\tr F=k.
	\]
By Ky Fan's variational principle,
\[
\tr\bigl(E\Phi_\Gamma(M)\bigr)
=\tr(FM)
\le
\sum_{j=1}^k\lambda_j^\downarrow(M).
\]
	Maximizing over all rank-\(k\) projections \(E\) gives the Ky Fan inequalities.
	For \(k=n\), equality follows from trace preservation.  Hence
	\[
	\lambda(\Gamma\circ M)\prec\lambda(M).
	\]
	
	We now prove \textup{(ii)}.  Set
	\[
	\widetilde{\Gamma}
	=
	\Gamma+\diag(1-\gamma_{11},\ldots,1-\gamma_{nn}).
	\]
	Then \(\widetilde{\Gamma}\ge0\) and \(\widetilde{\gamma}_{ii}=1\) for all \(i\).
	By \textup{(i)},
	\[
	\lambda(\widetilde{\Gamma}\circ M)\prec\lambda(M).
	\]
	Since \(M\ge0\), its diagonal entries are nonnegative, and
	\[
	\widetilde{\Gamma}\circ M-\Gamma\circ M
	=
	\diag\bigl((1-\gamma_{ii})M_{ii}\bigr)\ge0.
	\]
	Therefore, by Weyl monotonicity,
	\[
	\lambda(\Gamma\circ M)
	\prec_w
	\lambda(\widetilde{\Gamma}\circ M)
	\prec_w
	\lambda(M).
	\]
	This proves \textup{(ii)}.
\end{proof} 
 \begin{proposition}
 \label{prop:abstract-schur-reduction}
 Let $R,C\in\Pn$, let $a,b\ge0$, and let $0\le c\le2ab$.  Put
 \[
 S_{a,b}(R,C)=(aR^{-1}+bR)C(aR^{-1}+bR)
 \]
 and
 \[
 T_{a,b;c}(R,C)=a^2R^{-1}CR^{-1}+b^2RCR+cC.
 \]
 Then
 \[
 \lambda\bigl(T_{a,b;c}(R,C)\bigr)
 \prec_w
 \lambda\bigl(S_{a,b}(R,C)\bigr).
 \]
 If $c=2ab$, then
 \[
 \lambda\bigl(T_{a,b;2ab}(R,C)\bigr)
 \prec
 \lambda\bigl(S_{a,b}(R,C)\bigr).
 \]
 \end{proposition}
 
 \begin{proof}
 If $a=0$ or $b=0$, then $c=0$ and $T_{a,b;c}(R,C)=S_{a,b}(R,C)$.  Assume
 therefore that $a,b>0$.
 
 After a unitary conjugation, suppose that
 $R=\diag(r_1,\ldots,r_n)$ with $r_i>0$.  Set
 \[
 \alpha_i=ar_i^{-1},
 \qquad
 \beta_i=br_i,
 \qquad
 d_i=\alpha_i+\beta_i.
 \]
 Then $\alpha_i\beta_i=ab$.  Entrywise,
 \[
 \bigl(S_{a,b}(R,C)\bigr)_{ij}=d_id_jC_{ij},
 \]
 whereas
 \[
 \bigl(T_{a,b;c}(R,C)\bigr)_{ij}
 =(\alpha_i\alpha_j+\beta_i\beta_j+c)C_{ij}.
 \]
 Thus
 \[
 T_{a,b;c}(R,C)=\Gamma_c\circ S_{a,b}(R,C),
 \]
 where
 \[
 (\Gamma_c)_{ij}
 =\frac{\alpha_i\alpha_j+\beta_i\beta_j+c}{d_id_j}.
 \]
 Let $D=\diag(d_1,\ldots,d_n)$,
 $\alpha=(\alpha_1,\ldots,\alpha_n)^T$, and
 $\beta=(\beta_1,\ldots,\beta_n)^T$.  Then
 \[
 \Gamma_c
 =D^{-1}\bigl(\alpha\alpha^T+\beta\beta^T+c\one\one^T\bigr)D^{-1},
 \]
 so $\Gamma_c\ge0$.  Also,
 \[
 (\Gamma_c)_{ii}
 =\frac{\alpha_i^2+\beta_i^2+c}{(\alpha_i+\beta_i)^2}
 \le
 \frac{\alpha_i^2+\beta_i^2+2\alpha_i\beta_i}{(\alpha_i+\beta_i)^2}
 =1,
 \]
 because $c\le2ab=2\alpha_i\beta_i$.  Since \(S_{a,b}(R,C)\ge0\), the subcorrelation part of
 Lemma~\ref{lem:schur-subcorrelation} gives weak majorization.  If \(c=2ab\),
 then \((\Gamma_c)_{ii}=1\) for every \(i\), and the correlation-matrix part of
 Lemma~\ref{lem:schur-subcorrelation} gives majorization.
 \end{proof}
 
 \begin{remark}
 \label{rem:rank-three-multiplier}
 For every \(c\ge0\), the multiplier \(\Gamma_c\) appearing in
 Proposition~\ref{prop:abstract-schur-reduction} has rank at most three, since
 \[
 \Gamma_c
 =D^{-1}\bigl(\alpha\alpha^T+\beta\beta^T+c\one\one^T\bigr)D^{-1}.
 \]
 The endpoint \(c=2ab\) is special because the diagonal entries are then equal to
 one, so the multiplier is a correlation matrix.  More explicitly, in this case
 define
 \[
 s_i=\frac{\alpha_i-\beta_i}{d_i},
 \qquad
 t_i=\frac{2\sqrt{ab}}{d_i}.
 \]
 Then $s_i^2+t_i^2=1$ and
 \[
 (\Gamma_{2ab})_{ij}
 =\frac12(1+s_is_j+t_it_j).
 \]
 Therefore
 \[
 \Gamma_{2ab}=\frac12(\one\one^T+ss^T+tt^T),
 \]
 where $s=(s_1,\ldots,s_n)^T$ and $t=(t_1,\ldots,t_n)^T$.
 Thus, at the endpoint, \(\Gamma_{2ab}\) is a rank-at-most-three correlation
 matrix.
 \end{remark}
 
The spectral case admits a sharp coefficient form: the Wasserstein expression
dominates the quadratic spectral-Heron expression throughout the range
\(0\le c\le 2ab\).  At the endpoint \(c=2ab\), the comparison strengthens from
weak majorization to majorization.
 \begin{theorem}
 \label{thm:sharp-spectral-heron-wasserstein}
 Let $A,B\in\Pn$, let $a,b\ge0$, and let $0\le c\le2ab$.  Then
 \[
 \lambda\bigl(H^{\natural,c}_{a,b}(A,B)\bigr)
 \prec_w
 \lambda\bigl(W_{a,b}(A,B)\bigr).
 \]
 Consequently,
 \[
 \ui{H^{\natural,c}_{a,b}(A,B)}
 \le
 \ui{W_{a,b}(A,B)}
 \]
for every unitarily invariant norm.  For \(c=2ab\) one has the  
majorization
 \[
 \lambda\bigl(H^\natural_{a,b}(A,B)\bigr)
 \prec
 \lambda\bigl(W_{a,b}(A,B)\bigr).
 \]
 For fixed $a,b>0$, the upper coefficient $2ab$ is sharp among nonnegative
 coefficients, already in dimension one.
 \end{theorem}
 
 \begin{proof}
 If $a=0$ or $b=0$, then $c=0$ and the assertion is immediate.  Assume
 $a,b>0$.  Put
 \[
 X=A^{-1}\#B,
 \qquad
 R=X^{1/2},
 \qquad
 C=RAR.
 \]
 By Lemma~\ref{lem:riccati-parametrization}, $XAX=B$.  Hence
 \[
 A=R^{-1}CR^{-1},
 \qquad
 B=RCR,
 \qquad
 A\natural B=C.
 \]
 Furthermore,
 \[
 \begin{aligned}
 	W_{a,b}(A,B)
 	&=(aI+bX)A(aI+bX)                                      \\
 	&=(aR^{-1}+bR)C(aR^{-1}+bR),
 \end{aligned}
 \]
 whereas
 \[
 H^{\natural,c}_{a,b}(A,B)
 =a^2R^{-1}CR^{-1}+b^2RCR+cC.
 \]
 The majorization assertions follow from
 Proposition~\ref{prop:abstract-schur-reduction}.  The norm inequality follows
 from the Ky Fan dominance principle.
 
 If $c>2ab$, take $A=B=1$.  Then
 \[
 H^{\natural,c}_{a,b}(1,1)=a^2+b^2+c,
 \qquad
 W_{a,b}(1,1)=(a+b)^2.
 \]
 Weak majorization in dimension one is ordinary order.  A universal comparison
 would therefore force $a^2+b^2+c\le(a+b)^2$, equivalently $c\le2ab$.
 \end{proof}
 
 \begin{corollary}
 \label{cor:weighted-spectral-heron-wasserstein}
 
 Let $A,B\in\Pn$, let $0\le t\le1$, and let $0\le c\le2t(1-t)$.  Then
 \[
 \lambda\Bigl((1-t)^2A+t^2B+c(A\natural B)\Bigr)
 \prec_w
 \lambda(A\diamond_tB).
 \]
 Consequently,
 \[
 \ui{(1-t)^2A+t^2B+c(A\natural B)}
 \le
 \ui{A\diamond_tB}
 \]
 for every unitarily invariant norm.  At the endpoint coefficient \(c=2t(1-t)\), one has the  majorization
 \[
 \lambda\Bigl((1-t)^2A+t^2B+2t(1-t)(A\natural B)\Bigr)
 \prec
 \lambda(A\diamond_tB).
 \]
 \end{corollary}

The endpoint majorization also gives a useful spectral-spreading interpretation.
 \begin{corollary}
 \label{cor:spectral-spreading}
 Let $A,B\in\Pn$ and $a,b\ge0$.  Then, for every $1\le k\le n$,
 \[
 \sum_{j=1}^k\lambda_j^\downarrow\bigl(H^\natural_{a,b}(A,B)\bigr)
 \le
 \sum_{j=1}^k\lambda_j^\downarrow\bigl(W_{a,b}(A,B)\bigr),
 \]
 and, because the two matrices have the same trace,
 \[
 \sum_{j=n-k+1}^n\lambda_j^\downarrow\bigl(H^\natural_{a,b}(A,B)\bigr)
 \ge
 \sum_{j=n-k+1}^n\lambda_j^\downarrow\bigl(W_{a,b}(A,B)\bigr).
 \]
 In particular,
 \[
 \lambda_1^\downarrow\bigl(H^\natural_{a,b}(A,B)\bigr)
 \le
 \lambda_1^\downarrow\bigl(W_{a,b}(A,B)\bigr),
 \qquad
 \lambda_n^\downarrow\bigl(H^\natural_{a,b}(A,B)\bigr)
 \ge
 \lambda_n^\downarrow\bigl(W_{a,b}(A,B)\bigr).
 \]
 If $a,b>0$, then also
 \[
 \det H^\natural_{a,b}(A,B)
 \ge
 \det W_{a,b}(A,B).
 \]
 Thus the spectral Heron expression is spectrally less spread than the
 corresponding Bures-Wasserstein expression.
 \end{corollary}
\begin{proof}
	The inequalities for the sums of the \(k\) largest eigenvalues are exactly the
	Ky Fan inequalities from Theorem~\ref{thm:sharp-spectral-heron-wasserstein} in
	the case \(c=2ab\).  Since the endpoint comparison is a majorization, the two
	matrices have the same trace.  Hence the inequalities for the sums of the
	\(k\) smallest eigenvalues follow by subtracting the corresponding inequalities
	for the \(n-k\) largest eigenvalues from the common trace.
	
	If \(a,b>0\), then both \(H^\natural_{a,b}(A,B)\) and \(W_{a,b}(A,B)\) are
	positive definite.  Since
	\[
	\lambda(H^\natural_{a,b}(A,B))
	\prec
	\lambda(W_{a,b}(A,B)),
	\]
	Karamata's inequality applied to the concave function \(x\mapsto\log x\) on
	\((0,\infty)\) gives
	\[
	\sum_{j=1}^n\log\lambda_j(H^\natural_{a,b}(A,B))
	\ge
	\sum_{j=1}^n\log\lambda_j(W_{a,b}(A,B)).
	\]
	Exponentiating yields the determinant inequality.
\end{proof}
 
 \begin{proposition}
 \label{prop:matrix-equality-case}
 Let $A,B\in\Pn$ and let $a,b>0$.  Then
 \[
 H^\natural_{a,b}(A,B)=W_{a,b}(A,B)
 \]
 if and only if $A$ and $B$ commute.
 \end{proposition}
 
 \begin{proof}
 Use the notation from the proof of
 Theorem~\ref{thm:sharp-spectral-heron-wasserstein}.  After a unitary conjugation
 that diagonalizes $R$, write $R=\diag(r_1,\ldots,r_n)$.  Then
 \[
 \bigl(W_{a,b}(A,B)-H^\natural_{a,b}(A,B)\bigr)_{ij}
 =ab\left(\frac{r_i}{r_j}+\frac{r_j}{r_i}-2\right)C_{ij}
 =ab\frac{(r_i-r_j)^2}{r_ir_j}C_{ij}.
 \]
 Therefore equality holds if and only if $C_{ij}=0$ whenever $r_i\ne r_j$, which
 is exactly the condition $CR=RC$.  If $CR=RC$, then
 \[
 A=R^{-1}CR^{-1},
 \qquad
 B=RCR
 \]
 commute.  Conversely, if $A$ and $B$ commute, then $X=A^{-1}\#B$, $R=X^{1/2}$,
 and $C=RAR$ are simultaneously diagonalizable, so $CR=RC$.
 \end{proof}
 
 The result is not a Loewner-order comparison.  In the variables
 \(X=A^{-1}\#B\), \(R=X^{1/2}\), and \(C=RAR\),
 \[
 W_{a,b}(A,B)-H^\natural_{a,b}(A,B)
 =
 ab\bigl(R^{-1}CR+RCR^{-1}-2C\bigr).
 \]
 For \(R=\diag(1,2)\) and
 \(C=\begin{pmatrix}1&\rho\\ \rho&1\end{pmatrix}\), \(0<\rho<1\), the
 corresponding \(A=R^{-1}CR^{-1}\) and \(B=RCR\) are positive definite, while the
 difference above has zero diagonal and nonzero off-diagonal entries, hence is
 indefinite.  Thus the comparison is driven by Schur majorization rather than
 Loewner order.

 \section{Nonlinear pinching and direct Heron-Wasserstein comparisons}
 \label{sec:nonlinear-pinching}
 
 We now develop the nonlinear mechanism used for the Kubo-Ando Heron expression.
 
 \begin{lemma}
 \label{lem:nonlinear-kyfan-contraction}
 Let $\Phi:\Pn\to\Pn$ satisfy the following properties:
 \begin{enumerate}[label=\textup{(\roman*)}]
 	\item $\Phi$ is order preserving:
 	$C_1\le C_2$ implies $\Phi(C_1)\le\Phi(C_2)$.
 	\item $\Phi$ is positively homogeneous:
 	$\Phi(\alpha C)=\alpha\Phi(C)$ for $\alpha>0$.
 	\item $\Phi$ is unital: $\Phi(I)=I$.
 	\item $\Phi$ is trace-subpreserving: $\tr\Phi(C)\le\tr C$ for all $C\in\Pn$.
 \end{enumerate}
 Then
 \[
 \lambda(\Phi(C))\prec_w\lambda(C),
 \qquad C\in\Pn.
 \]
 \end{lemma}
 
 \begin{proof}
 Fix $C\in\Pn$.  For $t\ge0$, write $C=U\diag(c_1,\ldots,c_n)U^*$ and define
 \[
 C_t=C\vee tI:=U\diag\bigl(\max\{c_1,t\},\ldots,\max\{c_n,t\}\bigr)U^*.
 \]
 Then $C\le C_t$, and hence $\Phi(C)\le\Phi(C_t)$.  The functional
 $Y\mapsto\tr\pospart{Y-tI}$ is monotone in the Loewner order because
 \[
 \tr\pospart{Y-tI}
 =\max_{0\le E\le I}\tr\bigl(E(Y-tI)\bigr).
 \]
 Thus
 \[
 \tr\pospart{\Phi(C)-tI}
 \le
 \tr\pospart{\Phi(C_t)-tI}.
 \]
 If $t>0$, then $C_t\ge tI$.  By order preservation, homogeneity, and unitality,
 $\Phi(C_t)\ge tI$.  Therefore
 \[
 \begin{aligned}
 	\tr\pospart{\Phi(C)-tI}
 	&\le \tr\pospart{\Phi(C_t)-tI}     \\
 	&=\tr\Phi(C_t)-nt                  \\
 	&\le\tr C_t-nt                     \\
 	&=\tr\pospart{C-tI}.
 \end{aligned}
 \]
 For $t=0$, the same inequality is exactly trace-subpreservation.  Hence the
 inequality holds for every $t\ge0$.  Using the Ky Fan threshold formula
 \eqref{eq:kyfan-threshold-formula}, for $1\le k\le n$ we get
 \[
 \begin{aligned}
 	\sum_{j=1}^k\lambda_j^\downarrow(\Phi(C))
 	&=\min_{t\ge0}\left\{kt+\tr\pospart{\Phi(C)-tI}\right\} \\
 	&\le\min_{t\ge0}\left\{kt+\tr\pospart{C-tI}\right\}     \\
 	&=\sum_{j=1}^k\lambda_j^\downarrow(C).
 \end{aligned}
 \]
 Thus $\lambda(\Phi(C))\prec_w\lambda(C)$.
 \end{proof}
 For \(0<R<I\), put \(S=I-R\) and define
 \[
 \Phi_R(C)=RCR+SCS+2(RCR\#SCS),\qquad C\in\Pn.
 \]
 The direct Kubo-Ando comparison rests on the following nonlinear pinching
 estimate.
 
 \begin{theorem}
 \label{thm:nonlinear-pinching-majorization}
 Let $C\in\Pn$ and let $R\in\Pn$ satisfy $0<R<I$.  Then
 \[
 \lambda\bigl(\Phi_R(C)\bigr)
 \prec_w
 \lambda(C).
 \]
 Equivalently,
 \[
 \lambda\Bigl(RCR+(I-R)C(I-R)
 +2\bigl(RCR\#(I-R)C(I-R)\bigr)\Bigr)
 \prec_w
 \lambda(C).
 \]
 \end{theorem}
 
 \begin{proof}
 Set \(S=I-R\).  We verify the four hypotheses of
 Lemma~\ref{lem:nonlinear-kyfan-contraction} for the map \(\Phi_R\).  Order preservation follows from
 the monotonicity of congruence and of the Kubo-Ando geometric mean.  Positive
 homogeneity is immediate.  Since \(R\) and \(S\) commute,
 \[
 R^2\#S^2=RS,
 \]
 and therefore
 \[
 \Phi_R(I)=R^2+S^2+2RS=(R+S)^2=I.
 \]
 
 It remains to prove trace-subpreservation.  Put
 \[
 P=RCR,
 \qquad
 Q=SCS.
 \]
 Since \(R\) and \(S\) commute, the matrix
 \[
 Y=SR^{-1}=R^{-1}S
 \]
 is positive definite.  It satisfies \(YPY=Q\).  Indeed, although \(C\) need not
 commute with \(R\) and \(S\),
 \[
 (SR^{-1})(RCR)(SR^{-1})
 =SCRSR^{-1}
 =SC(RSR^{-1})
 =SCS.
 \]
 Hence, by the uniqueness of the positive solution of \(XPX=Q\),
 \[
 P^{-1}\#Q=SR^{-1}.
 \]
 Consequently,
 \[
 \begin{aligned}
 	P\natural Q
 	&=(P^{-1}\#Q)^{1/2}P(P^{-1}\#Q)^{1/2}        \\
 	&=(SR^{-1})^{1/2}RCR(SR^{-1})^{1/2}           \\
 	&=(RS)^{1/2}C(RS)^{1/2}.
 \end{aligned}
 \]
 Thus
 \[
 \tr(P\natural Q)=\tr(RSC).
 \]
 By \eqref{eq:kubo-spectral-trace},
 \[
 \tr(P\#Q)\le\tr(P\natural Q)=\tr(RSC).
 \]
 Therefore
 \[
 \begin{aligned}
 	\tr\Phi_R(C)
 	&=\tr(RCR)+\tr(SCS)+2\tr(RCR\#SCS)       \\
 	&\le \tr(R^2C)+\tr(S^2C)+2\tr(RSC)       \\
 	&=\tr\bigl((R+S)^2C\bigr)=\tr C.
 \end{aligned}
 \]
 Lemma~\ref{lem:nonlinear-kyfan-contraction} gives the desired weak
 majorization.
 \end{proof}
 The nonlinear pinching principle gives the corresponding comparison when the
 spectral cross term is replaced by the Kubo-Ando midpoint geometric mean.  For
 \(a,b\ge0\), set
 \[
 H^\#_{a,b}(A,B)
 :=
 a^2A+b^2B+2ab(A\#B).
 \]
 
 \begin{theorem}
 \label{thm:direct-heron-wasserstein}
 Let $A,B\in\Pn$ and let $a,b\ge0$.  Then
 \[
 \lambda\bigl(H^\#_{a,b}(A,B)\bigr)
 \prec_w
 \lambda\bigl(W_{a,b}(A,B)\bigr).
 \]
 Consequently,
 \[
 \ui{H^\#_{a,b}(A,B)}
 \le
 \ui{W_{a,b}(A,B)}
 \]
 for every unitarily invariant norm.
 \end{theorem}
 \begin{proof}
 The result is immediate if $a=0$ or $b=0$.  Assume $a,b>0$ and put
 \[
 X=A^{-1}\#B,
 \qquad
 T=aI+bX.
 \]
 By Lemma~\ref{lem:riccati-parametrization},
 \[
 W_{a,b}(A,B)=TAT.
 \]
 Set
 \[
 C=TAT,
 \qquad
 R=aT^{-1},
 \qquad
 S=bXT^{-1}.
 \]
 Since \(X\) commutes with \(T=aI+bX\), we have
 \[
 R,S>0,
 \qquad
 RS=SR,
 \qquad
 R+S=I.
 \]
 In particular, \(0<R<I\) and \(S=I-R\).  Moreover,
 \[
 RCR=a^2T^{-1}(TAT)T^{-1}=a^2A.
 \]
 Again using the commutativity of \(X\) and \(T\),
 \[
 \begin{aligned}
 	SCS
 	&=b^2XT^{-1}(TAT)XT^{-1}        \\
 	&=b^2XATXT^{-1}                 \\
 	&=b^2XA(TXT^{-1})                \\
 	&=b^2XAX=b^2B.
 \end{aligned}
 \]
 By homogeneity of the geometric mean,
 \[
 RCR\#SCS=(a^2A)\#(b^2B)=ab(A\#B).
 \]
 Hence
 \[
 H^\#_{a,b}(A,B)=RCR+SCS+2(RCR\#SCS)=\Phi_R(C).
 \]
 Applying Theorem~\ref{thm:nonlinear-pinching-majorization} to \(C\) and \(R\),
 we obtain
 \[
 \lambda\bigl(H^\#_{a,b}(A,B)\bigr)
 =\lambda\bigl(\Phi_R(C)\bigr)
 \prec_w
 \lambda(C)
 =\lambda\bigl(W_{a,b}(A,B)\bigr).
 \]
 The norm inequality follows from Ky Fan dominance.
 \end{proof}

\begin{proposition}
	\label{prop:kubo-matrix-equality-case}
	Let \(A,B\in\Pn\) and let \(a,b>0\).  Then
	\[
	H^\#_{a,b}(A,B)=W_{a,b}(A,B)
	\]
	if and only if \(A\) and \(B\) commute.
\end{proposition}

\begin{proof}
	If \(A\) and \(B\) commute, then
	\[
	(AB)^{1/2}=(BA)^{1/2}=A^{1/2}B^{1/2}=A\#B,
	\]
	and hence \(H^\#_{a,b}(A,B)=W_{a,b}(A,B)\).
	
	Conversely, assume that
	\[
	H^\#_{a,b}(A,B)=W_{a,b}(A,B).
	\]
	Since \(a,b>0\), this is equivalent to
	\[
	2(A\#B)=(AB)^{1/2}+(BA)^{1/2}.
	\]
	Put
	\[
	X=A^{-1}\#B.
	\]
	By Lemma~\ref{lem:riccati-parametrization}, \(XAX=B\),
	\[
	(AB)^{1/2}=AX,
	\qquad
	(BA)^{1/2}=XA.
	\]
	Set
	\[
	Y=A^{1/2}XA^{-1/2}.
	\]
	Then
	\[
	AX=A^{1/2}YA^{1/2},
	\qquad
	XA=A^{1/2}Y^*A^{1/2}.
	\]
	Moreover,
	\[
	A^{-1/2}BA^{-1/2}
	=
	A^{-1/2}XAXA^{-1/2}
	=
	Y^*Y,
	\]
	and therefore
	\[
	A\#B=A^{1/2}|Y|A^{1/2}.
	\]
	Thus the equality
	\[
	2(A\#B)=AX+XA
	\]
	is equivalent, after congruence by \(A^{-1/2}\), to
	\[
	2|Y|=Y+Y^*.
	\]
	Let \(Y=U|Y|\) be the polar decomposition of \(Y\).  Since \(Y\) is invertible,
	\(U\) is unitary.  Taking traces in the last identity gives
	\[
	\tr |Y|=\operatorname{Re}\tr(U|Y|).
	\]
	On the other hand,
	\[
	\tr |Y|-\operatorname{Re}\tr(U|Y|)
	=
	\tr\bigl(|Y|^{1/2}(I-\operatorname{Re}U)|Y|^{1/2}\bigr).
	\]
	For a unitary \(U\),
\[
I-\operatorname{Re}U=\frac12(I-U)^*(I-U)\ge0.
\]
Hence the last trace is nonnegative.  Equality forces
\[
|Y|^{1/2}(I-\operatorname{Re}U)|Y|^{1/2}=0.
\]
As \(|Y|>0\), this gives \(\operatorname{Re}U=I\), and the displayed identity
above yields \(U=I\).  Hence \(Y=|Y|\) is Hermitian.
	
	Thus
	\[
	A^{1/2}XA^{-1/2}=A^{-1/2}XA^{1/2},
	\]
	which is equivalent to \(AX=XA\).  Since \(B=XAX\), it follows that \(A\) and
	\(B\) commute.
\end{proof}
 
The direct comparison also has a lower-endpoint consequence which is not part of
weak majorization alone.  This endpoint estimate is independent of weak
majorization, which controls only upper Ky Fan sums.
 
 \begin{proposition}
 \label{prop:direct-endpoints}
 Let $A,B\in\Pn$ and let $a,b\ge0$.  Then
 \[
 \lambda_n^\downarrow\bigl(H^\#_{a,b}(A,B)\bigr)
 \ge
 \lambda_n^\downarrow\bigl(W_{a,b}(A,B)\bigr).
 \]
 Consequently, together with Theorem~\ref{thm:direct-heron-wasserstein},
 \[
 \lambda_1^\downarrow\bigl(H^\#_{a,b}(A,B)\bigr)
 \le
 \lambda_1^\downarrow\bigl(W_{a,b}(A,B)\bigr),
 \qquad
 \tr H^\#_{a,b}(A,B)
 \le
 \tr W_{a,b}(A,B),
 \]
 and
 \[
 \sum_{j=1}^{n-1}\lambda_j^\downarrow
 \bigl(H^\#_{a,b}(A,B)\bigr)
 \le
 \sum_{j=1}^{n-1}\lambda_j^\downarrow
 \bigl(W_{a,b}(A,B)\bigr).
 \]
 \end{proposition}
 
 \begin{proof}
 The assertion is trivial if $a=0$ or $b=0$, so assume $a,b>0$.  Put
 \[
 X=A^{-1}\#B,
 \qquad
 T=aI+bX.
 \]
 Then \(XAX=B\).  Equivalently, using the Riccati identities
 \((AB)^{1/2}=AX\) and \((BA)^{1/2}=XA\), one has
 \[
 W_{a,b}(A,B)=TAT.
 \]
 
 By homogeneity, it is enough to show
 \[
 W_{a,b}(A,B)\ge I
 \quad\Longrightarrow\quad
 H^\#_{a,b}(A,B)\ge I.
 \]
 If \(TAT\ge I\), then congruence by \(T^{-1}\) gives
 \[
 A\ge T^{-2},
 \qquad
 B=XAX\ge XT^{-2}X.
 \]
 The map
 \[
 (P,Q)\mapsto a^2P+b^2Q+2ab(P\#Q)
 \]
 is monotone in both variables.  Hence
 \[
 H^\#_{a,b}(A,B)
 \ge
 H^\#_{a,b}(T^{-2},XT^{-2}X).
 \]
 Since \(T=aI+bX\) commutes with \(X\),
 \[
 T^{-2}\#XT^{-2}X=XT^{-2}.
 \]
 Therefore
 \[
 \begin{aligned}
 	H^\#_{a,b}(T^{-2},XT^{-2}X)
 	&=a^2T^{-2}+b^2X^2T^{-2}+2abXT^{-2}        \\
 	&=T^{-2}(aI+bX)^2=I.
 \end{aligned}
 \]
 Thus \(W_{a,b}(A,B)\ge I\) implies \(H^\#_{a,b}(A,B)\ge I\).  Let
 \[
 m=\lambda_n^\downarrow\bigl(W_{a,b}(A,B)\bigr)>0.
 \]
 By homogeneity,
 \[
 W_{a,b}(A/m,B/m)=m^{-1}W_{a,b}(A,B)\ge I.
 \]
 Applying the implication just proved to $A/m$ and $B/m$ gives
 \[
 m^{-1}H^\#_{a,b}(A,B)
 =H^\#_{a,b}(A/m,B/m)
 \ge I.
 \]
 Therefore
 \[
 \lambda_n^\downarrow\bigl(H^\#_{a,b}(A,B)\bigr)
 \ge
 \lambda_n^\downarrow\bigl(W_{a,b}(A,B)\bigr).
 \]
 The remaining displayed inequalities follow from
 Theorem~\ref{thm:direct-heron-wasserstein}.
 \end{proof}
 
 The same comparison remains valid if the Kubo-Ando cross term is given any
 nonnegative coefficient up to the Heron coefficient \(2ab\).
 
 \begin{remark}
 The same weak-majorization comparison holds with \(2ab(A\#B)\) replaced by
 \(c(A\#B)\) for every \(0\le c\le2ab\).  Indeed,
 \[
 a^2A+b^2B+c(A\#B)\le H^\#_{a,b}(A,B),
 \]
 and the claim follows from Weyl monotonicity and
 Theorem~\ref{thm:direct-heron-wasserstein}.  The coefficient \(2ab\) is sharp
 already for \(A=B=1\).
 \end{remark}

\subsection{Incomparability of the two Heron expressions}
\label{subsec:heron-incomparability}
 
The two Heron expressions are not directly comparable in weak majorization.  The spectral result
and the Kubo-Ando result both compare their respective Heron expressions with
\(W_{a,b}(A,B)\), but neither Heron expression weakly majorizes the other in
general.  Thus the Kubo-Ando comparison above cannot be obtained from the
spectral comparison by ordering the two Heron cross terms inside the same
additive expression.

\begin{proposition}
	\label{prop:heron-expressions-incomparable}
	There is no universal weak-majorization comparison, in either direction, between
	the Kubo-Ando and spectral Heron expressions.  This failure occurs already for
	\(a=b=1\): neither
	\[
	\lambda\bigl(A+B+2(A\#B)\bigr)
	\prec_w
	\lambda\bigl(A+B+2(A\natural B)\bigr)
	\]
	nor
	\[
	\lambda\bigl(A+B+2(A\natural B)\bigr)
	\prec_w
	\lambda\bigl(A+B+2(A\#B)\bigr)
	\]
	holds for all \(A,B\in\Pn\).
\end{proposition}
 
 \begin{proof}
The first example below is certified by exact rational computations.  Let
 	\[
 	A=\diag(1,20,40)
 	\]
 	and
 	\[
 	R=\frac1{20}
 	\begin{pmatrix}
 		42&-4&2\\
 		-4&6&2\\
 		2&2&1
 	\end{pmatrix}.
 	\]
 	The leading principal minors of \(R\) are
 	\[
 	\frac{21}{10},\qquad
 	\frac{59}{100},\qquad
 	\frac{3}{2000},
 	\]
 	so \(R>0\).  Put
 	\[
 	X=R^2,\qquad B=XAX.
 	\]
 	Then \(X>0\), \(B>0\), and \(XAX=B\).  Hence, by the Riccati characterization,
 	\(X=A^{-1}\#B\), and because \(R=X^{1/2}\),
 	\[
 	A\natural B=X^{1/2}AX^{1/2}=RAR.
 	\]
 	For reference,
 	\[
 	B=
 	\begin{pmatrix}
 		\frac{129153}{5000} & -\frac{4119}{1250} & \frac{4521}{5000}\\
 		-\frac{4119}{1250} & \frac{6219}{10000} & -\frac{723}{20000}\\
 		\frac{4521}{5000} & -\frac{723}{20000} & \frac{2511}{40000}
 	\end{pmatrix}.
 	\]
 	
 	We shall prove the strict separation
 	\begin{equation}
 		\label{eq:clean-rational-separation-additive-obstruction}
 		\lambda_1^\downarrow\bigl(A+B+2(A\#B)\bigr)
 		>
 		41
 		>
 		\lambda_1^\downarrow\bigl(A+B+2(A\natural B)\bigr).
 	\end{equation}
 	For the lower bound, set
 	\[
 	G=\frac1{10000}
 	\begin{pmatrix}
 		50041&-6271&1796\\
 		-6271&19846&7265\\
 		1796&7265&3009
 	\end{pmatrix}.
 	\]
 	We use the standard block-matrix maximal characterization of the geometric mean
(see, e.g., \cite[Chapter~4]{Bhatia2007}):
 	\[
 	A\#B
 	=
 	\max\left\{
 	Z=Z^*:
 	\begin{pmatrix}A&Z\\ Z&B\end{pmatrix}\ge0
 	\right\},
 	\]
 	where the maximum is taken with respect to the Loewner order.  Thus it is
 	enough to certify
 	\[
 	\begin{pmatrix}A&G\\G&B\end{pmatrix}\ge0.
 	\]
 	Since \(A>0\), this is equivalent to
 	\[
 	K:=B-GA^{-1}G\ge0.
 	\]
 	A direct exact calculation gives the leading principal minors
 	\[
 	\begin{aligned}
 		\Delta_1(K)
 		&=
 		\frac{1538228131}{2\cdot10^9},\\
 		\Delta_2(K)
 		&=
 		\frac{36854830002529581}{8\cdot10^{18}},\\
 		\Delta_3(K)
 		&=
 		\frac{249208941796751}{2\cdot10^{26}}.
 	\end{aligned}
 	\]
 	Since \(K\) is Hermitian, these positive leading principal minors imply
 	\(K>0\) by Sylvester's criterion.  Consequently \(G\le A\#B\).
 	
 	Put
 	\[
 	M_G=A+B+2G.
 	\]
 	Then
 	\[
 	\det(41I-M_G)
 	=
 	-\frac{2296964316553}{10^{12}}<0.
 	\]
 	Since \(M_G\) is Hermitian, the matrix \(41I-M_G\) is Hermitian.  Its
 	negative determinant implies that \(41I-M_G\) has a negative eigenvalue, and
 	therefore
 	\[
 	\lambda_1^\downarrow(M_G)>41.
 	\]
 	Since \(G\le A\#B\), Weyl monotonicity gives
 	\[
 	\lambda_1^\downarrow\bigl(A+B+2(A\#B)\bigr)
 	\ge
 	\lambda_1^\downarrow(M_G)
 	>
 	41.
 	\]
 	
 	It remains to prove the upper bound in
 	\eqref{eq:clean-rational-separation-additive-obstruction}.  Since
 	\(A\natural B=RAR\), put
 	\[
 	M_\natural=A+B+2RAR.
 	\]
 	A direct exact calculation gives the leading principal minors of
 	\(41I-M_\natural\):
 	\[
 	\begin{aligned}
 		\Delta_1(41I-M_\natural)
 		&=
 		\frac{14747}{5000},\\
 		\Delta_2(41I-M_\natural)
 		&=
 		\frac{139973371}{10000000},\\
 		\Delta_3(41I-M_\natural)
 		&=
 		\frac{3328064679}{2\cdot10^{10}}.
 	\end{aligned}
 	\]
 	Since \(41I-M_\natural\) is Hermitian, these positive leading principal minors
 	imply \(41I-M_\natural>0\) by Sylvester's criterion.  Therefore	
 	\[
 	\lambda_1^\downarrow\bigl(A+B+2(A\natural B)\bigr)
 	=
 	\lambda_1^\downarrow(M_\natural)
 	<
 	41.
 	\]
Combining the two inequalities proves
\eqref{eq:clean-rational-separation-additive-obstruction}.  Hence
\[
\lambda\bigl(A+B+2(A\#B)\bigr)
\prec_w
\lambda\bigl(A+B+2(A\natural B)\bigr)
\]
fails already for \(k=1\).
For the opposite direction, it is enough to use a \(2\times2\) example and the
trace.  Let
\[
A=\diag(1,4),
\qquad
X=
\begin{pmatrix}
	1&1/2\\
	1/2&1
\end{pmatrix},
\qquad
B=XAX.
\]
Then
\[
B=
\begin{pmatrix}
	2&5/2\\
	5/2&17/4
\end{pmatrix},
\]
and \(X>0\), \(B>0\), \(XAX=B\).  Hence \(X=A^{-1}\#B\), and
\[
A\natural B=X^{1/2}AX^{1/2}.
\]
Thus
\[
\tr(A\natural B)=\tr(AX)=5.
\]

It remains to compute \(\tr(A\#B)\).  Put
\[
D=A^{-1/2}BA^{-1/2}
=
\begin{pmatrix}
	2&5/4\\
	5/4&17/16
\end{pmatrix}.
\]
Then
\[
\tr D=\frac{49}{16},
\qquad
\det D=\frac{9}{16}.
\]
For a \(2\times2\) positive definite matrix \(D\),
\[
D^{1/2}
=
\frac{D+\sqrt{\det D}\,I}
{\sqrt{\tr D+2\sqrt{\det D}}}.
\]
This follows from the Cayley-Hamilton identity, since
\[
\bigl(D+\sqrt{\det D}\,I\bigr)^2
=
\bigl(\tr D+2\sqrt{\det D}\bigr)D.
\]
Since \(\sqrt{\det D}=3/4\), we get
\[
D^{1/2}
=
\frac{D+\frac34 I}{\frac{\sqrt{73}}4}.
\]
Therefore
\[
\begin{aligned}
	\tr(A\#B)
	&=
	\tr\bigl(A^{1/2}D^{1/2}A^{1/2}\bigr)       \\
	&=
	\tr(AD^{1/2})                               \\
	&=
	\frac{\tr(AD)+\frac34\tr A}{\frac{\sqrt{73}}4}
	=
	\frac{40}{\sqrt{73}}.
\end{aligned}
\]
Since \(40/\sqrt{73}<5\), we have
\[
\tr(A\#B)<\tr(A\natural B).
\]
Consequently,
\[
\begin{aligned}
	\tr\bigl(A+B+2(A\#B)\bigr)
	&<
	\tr\bigl(A+B+2(A\natural B)\bigr).
\end{aligned}
\]
Thus
\[
\lambda\bigl(A+B+2(A\natural B)\bigr)
\prec_w
\lambda\bigl(A+B+2(A\#B)\bigr)
\]
cannot hold, because weak majorization at \(k=n\) would imply the opposite trace
inequality.
\end{proof}
 
This incomparability explains why the Kubo-Ando Heron-Wasserstein comparison
is not a consequence of the spectral Heron-Wasserstein comparison.  Although
\(A\#B\) and \(A\natural B\) are related by log-majorization and other
multiplicative comparisons, such information is not stable under adding the
common term \(a^2A+b^2B\).

 \section{The two-variable Bhatia-Lim-Yamazaki Heron inequality}
 \label{sec:bly}
 
 The nonlinear pinching theorem also yields a weak-majorization strengthening of
 the two-variable Heron inequality considered by Bhatia, Lim, and Yamazaki.  The
 additional ingredient is a lifting lemma for weak majorization under quadratic
 congruence.
 
 \begin{lemma}
 \label{lem:quadratic-lifting-weak-majorization}
 Let $C,D\in\Pn$.  If
 \[
 \lambda(D)\prec_w\lambda(C),
 \]
 then
 \[
 \lambda(C^{1/2}DC^{1/2})
 \prec_w
 \lambda(C^2).
 \]
 \end{lemma}
 
 \begin{proof}
 Fix $1\le k\le n$, and let $E$ be an orthogonal projection of rank $k$.  Set
 \[
 F=C^{1/2}EC^{1/2}.
 \]
 Then $0\le F\le C$ and $\rank F\le k$.  By von Neumann's trace inequality,
 \[
 \tr(FD)
 \le
 \sum_{j=1}^k\lambda_j^\downarrow(F)\lambda_j^\downarrow(D).
 \]
 Since $0\le F\le C$, the min-max principle gives
 \[
 \lambda_j^\downarrow(F)\le\lambda_j^\downarrow(C),
 \qquad 1\le j\le k.
 \]
 Since \(D\ge0\), the numbers \(\lambda_j^\downarrow(D)\) are nonnegative.
 Therefore
 \[
 \tr(FD)
 \le
 \sum_{j=1}^k\lambda_j^\downarrow(C)\lambda_j^\downarrow(D).
 \]
 Put
 $
 x_j=\lambda_j^\downarrow(C),
 y_j=\lambda_j^\downarrow(D).
 $
 The assumption $\lambda(D)\prec_w\lambda(C)$ says that
 \[
 \Delta_m:=\sum_{j=1}^m(y_j-x_j)\le0,
 \qquad 1\le m\le n.
 \]
 Since $x_1\ge\cdots\ge x_n\ge0$, Abel summation gives
 \[
 \begin{aligned}
 	\sum_{j=1}^k x_j(y_j-x_j)
 	&=x_k\Delta_k+
 	\sum_{m=1}^{k-1}(x_m-x_{m+1})\Delta_m
 	\le0.
 \end{aligned}
 \]
 Hence
 \[
 \sum_{j=1}^k\lambda_j^\downarrow(C)\lambda_j^\downarrow(D)
 \le
 \sum_{j=1}^k\lambda_j^\downarrow(C)^2
 =
 \sum_{j=1}^k\lambda_j^\downarrow(C^2).
 \]
 Thus, for every rank-$k$ projection $E$,
\[
\tr\bigl(E C^{1/2}DC^{1/2}\bigr)
=\tr(FD)
\le
\sum_{j=1}^k\lambda_j^\downarrow(C^2).
\]
 Maximizing over all such $E$ gives the desired Ky Fan inequalities.
 \end{proof}
 
 \begin{corollary}
 \label{cor:quadratically-lifted-nonlinear-pinching}
 Let $C\in\Pn$, and let $R,S\in\Pn$ commute and satisfy $R+S=I$.  Then
 \[
 \lambda\Bigl(C^{1/2}\bigl(RCR+SCS+2(RCR\#SCS)\bigr)C^{1/2}\Bigr)
 \prec_w
 \lambda(C^2).
 \]
 \end{corollary}
 
 \begin{proof}
 By Theorem~\ref{thm:nonlinear-pinching-majorization},
 \[
 \lambda\bigl(RCR+SCS+2(RCR\#SCS)\bigr)\prec_w\lambda(C).
 \]
 Apply Lemma~\ref{lem:quadratic-lifting-weak-majorization} with
 $D=RCR+SCS+2(RCR\#SCS)$.
 \end{proof}
 
 We now obtain the two-variable Heron form of the question considered by
 Bhatia, Lim, and Yamazaki.  The result is proved in the stronger
 weak-majorization form.
 
 \begin{theorem}
 \label{thm:bly-heron-weak-majorization}
 
 Let $A,B\in\Pn$ and let $a,b\ge0$.  Then
 \[
 \lambda\bigl(a^2A+b^2B+2ab(A\#B)\bigr)
 \prec_w
 \lambda\bigl(a^2A+b^2B
 +ab(A^{1/2}B^{1/2}+B^{1/2}A^{1/2})\bigr).
 \]
 Consequently,
 \[
 \ui{a^2A+b^2B+2ab(A\#B)}
 \le
 \ui{(aA^{1/2}+bB^{1/2})^2}
 \]
 for every unitarily invariant norm.  In particular,
 \[
 \ui{A+B+2(A\#B)}
 \le
 \ui{A+B+A^{1/2}B^{1/2}+B^{1/2}A^{1/2}}.
 \]
 \end{theorem}
 
 \begin{proof}
 The assertion is immediate if $a=0$ or $b=0$.  Assume $a,b>0$ and put
 \[
 Y=aA^{1/2},
 \qquad
 Z=bB^{1/2},
 \qquad
 T=Y+Z.
 \]
 Then $T\in\Pn$.  Define
 \[
 R=T^{-1/2}YT^{-1/2},
 \qquad
 S=T^{-1/2}ZT^{-1/2}.
 \]
 Then $R,S>0$ and $R+S=I$, so $R$ and $S$ commute.  Moreover,
 \[
 Y=T^{1/2}RT^{1/2},
 \qquad
 Z=T^{1/2}ST^{1/2}.
 \]
 Therefore
 \[
 Y^2=T^{1/2}RTRT^{1/2},
 \qquad
 Z^2=T^{1/2}STST^{1/2}.
 \]
 By congruence invariance of the Kubo-Ando geometric mean,
 \[
 Y^2\#Z^2
 =T^{1/2}\bigl(RTR\#STS\bigr)T^{1/2}.
 \]
 Since
 \[
 Y^2=a^2A,
 \qquad
 Z^2=b^2B,
 \qquad
 Y^2\#Z^2=ab(A\#B),
 \]
 it follows that
 \[
 \begin{aligned}
 	a^2A+b^2B+2ab(A\#B)
 	&=Y^2+Z^2+2(Y^2\#Z^2)                              \\
 	&=T^{1/2}\bigl(RTR+STS+2(RTR\#STS)\bigr)T^{1/2}.
 \end{aligned}
 \]
 Applying Corollary~\ref{cor:quadratically-lifted-nonlinear-pinching} with
 $C=T$, we obtain
 \[
 \lambda\bigl(a^2A+b^2B+2ab(A\#B)\bigr)
 \prec_w
 \lambda(T^2).
 \]
 Finally,
 \[
 T^2=(aA^{1/2}+bB^{1/2})^2.
 \]
 The norm inequality follows from Ky Fan dominance.
 \end{proof}
 
 \begin{remark}
 \label{rem:semidefinite-extension-scope}
By replacing \(A,B\) with \(A+\varepsilon I,B+\varepsilon I\), applying the
theorem to the positive definite pair, and then letting
\(\varepsilon\downarrow0\), the same conclusion extends to positive
semidefinite matrices.  This passage to the limit is justified by continuity of
the square root, of the Kubo-Ando geometric mean, and of Ky Fan sums.  The
theorem proves the two-variable Heron form in the weak-majorization sense.  It
should not be stated as a solution of the full multivariable power-mean problem
without a separate reduction.
 \end{remark}

\end{document}